\documentclass[11pt]{amsart}
\usepackage{amsmath}
\usepackage{amssymb,amsmath,amsthm}
\newtheorem{proposition}{Proposition}[section]
\newtheorem{theorem}[proposition]{Theorem}
\newtheorem{corollary}[proposition]{Corollary}
\newtheorem{lemma}[proposition]{Lemma}
\theoremstyle{definition}
\newtheorem{definition}[proposition]{Definition}
\newtheorem{remark}[proposition]{Remark}

\numberwithin{equation}{section}

\def\e{{\rm e}}
\def\eps{\varepsilon}
\def\th{\vartheta}
\def\d{{\rm d}}
\def\dist{{\rm dist}}

\def\AA {{\mathfrak A}}
\def\EE {{\mathfrak E}}
\def\BB {{\mathfrak B}}
\def\LL {{\mathfrak L}}
\def\II {{\mathfrak I}}
\def\t {{\mathbb T}}
\def\Ball{{B}}
\def\R {\mathbb{R}}
\def\N {\mathbb{N}}
\def\H {{\mathcal H}}
\def\V {{\mathcal V}}

\def\B {{\mathcal B}}
\def\C {{\mathcal C}}
\def\D {{\rm dom}}

\def\K {{\mathcal K}}

\def\M {{\mathcal M}}
\def\W {{\mathcal W}}

\def\Z {{\mathcal Z}}
\def\X {{\mathcal X}}
\def\Y {{\mathcal Y}}
\def \l {\langle}
\def \r {\rangle}
\def \pt {\partial_t}

\def \and{\quad\text{and}\quad}

\def \au {\rm}
\def \ti {\it}
\def \jou {\rm}
\def \bk {\rm}
\def \no#1#2#3 {{\bf #1} (#3), #2.}
\def \eds#1#2#3 {#1, #2, #3.}
\title[Asymptotics of the Coleman-Gurtin model]
{Asymptotics of the Coleman-Gurtin model}

\author[M.D. Chekroun, F. Di Plinio, N.E. Glatt-Holtz and V. Pata]
{}


\thanks{Research partially supported by the United States National Science Foundation
under grant NSF-DMS-0604235.}

\email{chekro@lmd.ens.fr {\rm (M.D. Chekroun)}}
\email{fradipli@indiana.edu {\rm (F. Di Plinio)}}
\email{negh@indiana.edu {\rm (N.E. Glatt-Holtz))}}
\email{vittorino.pata@polimi.it {\rm (V. Pata)}}

\begin{document}
\maketitle

%\medskip

\centerline{\scshape M.D. Chekroun}
\medskip
{\footnotesize
 \centerline{\'Ecole Normale Sup\'erieure, CERES-ERTI}
   \centerline{75231 Paris Cedex 05, France}
}

\medskip

\centerline{\scshape F. Di Plinio and N.E. Glatt-Holtz}
\medskip
{\footnotesize
 \centerline{Indiana University, Mathematics Department}
   \centerline{Bloomington, IN 47405, USA}
}

\medskip

\centerline{\scshape Vittorino Pata}
\medskip
{\footnotesize
 \centerline{Politecnico di Milano, Dipartimento di Matematica ``F.\ Brioschi''}
   \centerline{20133 Milano, Italy}
}

\bigskip

\begin{abstract}
This paper is concerned with the integrodifferential equation
$$\pt u-\Delta  u -\int_0^\infty \kappa(s)\Delta u(t-s)\,\d s +
\varphi(u)=f$$ arising in the Coleman-Gurtin's theory of heat
conduction with hereditary memory, in presence of a nonlinearity
$\varphi$ of critical growth. Rephrasing the equation within the
history space framework, we prove the existence of global and
exponential attractors of optimal regularity and finite fractal
dimension for the related solution semigroup, acting both on the
basic weak-energy space and on a more regular phase space.
\end{abstract}

%%%%%%%%%%%%%%%%%%%%%%%%%%%%%%%%%%%%%%%%%%%%
\section{Introduction}

\subsection{The model equation}

Let $\Omega\subset\R^3$ be a bounded domain with a sufficiently
smooth boundary $\partial\Omega$. For $t>0$, we consider the
integrodifferential equation in the variable
$u=u(\boldsymbol{x},t):\Omega\times\R\to\R$
\begin{equation}
\label{ORI} \pt u-\Delta  u -\int_0^\infty \kappa(s)\Delta
u(t-s)\,\d s + \varphi(u)=f,
\end{equation}
subject to the Dirichlet boundary condition
\begin{equation}
\label{ORI-BC}
u(\boldsymbol{x},t)_{|\boldsymbol{x}\in\partial\Omega}=0.
\end{equation}
The function $u$ is supposed to be known for all $t\leq 0$.
Accordingly, the boundary-value problem \eqref{ORI}-\eqref{ORI-BC}
is supplemented with the ``initial condition"
\begin{equation}
\label{ORI-IC} u(\boldsymbol{x},t)={\hat
u}(\boldsymbol{x},t),\quad\forall t\leq 0,
\end{equation}
where ${\hat u}:\Omega\times(-\infty,0]\to\R$ is a given function
accounting for the {\it initial past history} of $u$. In the sequel,
we agree to omit the dependence on $\boldsymbol{x}\in\Omega$.

Among many other diffusive phenomena, equation~\eqref{ORI} models
heat propagation in a homogeneous isotropic heat conductor with
hereditary memory. Here, the classical Fourier law ruling the heat
flux is replaced by the more physical constitutive relation devised
in the seminal paper of B.D.\ Coleman and M.E.\ Gurtin~\cite{COL},
based on the key assumption that the heat flux evolution is
influenced by the past history of the temperature gradient (see also
\cite{Amnesia,GG,GMP,GRA,LON,MIL,NUN}). In that case, $u$ represents
the temperature variation field relative to the equilibrium
reference value, $f$ is a time-independent external heat supply, and
the nonlinear term $\varphi(u)$ has to comply with some
dissipativity assumptions, although it can exhibit an
antidissipative behavior at low temperatures. Such a nonlinearity is
apt to describe, for instance, temperature-dependent radiative
phenomena (cf.\ \cite{LYS}).

\subsection{Basic assumptions}

We take $f\in L^2(\Omega)$ and $\varphi\in\C^2(\R)$, with
$\varphi(0)=0$, satisfying the growth and the dissipation conditions
\begin{align}
\label{critical}
|\varphi''(u)|&\leq c(1+|u|^p),\quad p\in[0,3],\\
\label{DERIVATA} \liminf_{|u|\to\infty}\varphi'(u)&> -\lambda_1,
\end{align}
where $\lambda_1>0$ is the first eigenvalue of the Laplace-Dirichlet
operator on $L^2(\Omega)$. Concerning the {\it memory kernel}, we
assume
$$\kappa(s)=\kappa_0-\int_0^s\mu(\sigma)\,\d\sigma,\quad\kappa_0>0,$$
for some (nonnegative) nonincreasing summable function $\mu$ on
$\R^+=(0,\infty)$ of total mass
$$\int_0^\infty \mu(s)\,\d s=\kappa_0.$$
Consequently, $\kappa$ is nonincreasing and nonnegative. Moreover,
we require the inequality (cf.\ \cite{GMPZ})
\begin{equation}
\label{NEC} \kappa(s)\leq\Theta\mu(s)
\end{equation}
to hold for every $s>0$ and some $\Theta>0$. Observe that
\eqref{NEC} implies the exponential decay
$$\kappa(s)\leq \kappa_0\e^{-s/\Theta}.$$
As a byproduct, $\kappa$ is summable on $\R^+$. To avoid the
presence of unnecessary constants, we agree to put
$$
\int_0^\infty \kappa(s)\,\d s=\int_0^\infty s\mu(s)\,\d s=1,
$$
where the first equality follows from an integration by parts.

\subsection{Asymptotic behavior}

The present paper is focused on the asymptotic properties of the
solutions to \eqref{ORI}-\eqref{ORI-IC}. Setting the problem in the
so-called history space framework~\cite{DAF} (see the next
Section~\ref{SSemi}), in order to have a solution semigroup, our
goal is to obtain global and exponential attractors of optimal
regularity and finite fractal dimension. We address the reader to
the books \cite{BV,CV,HAL,HAR,LAD,MZ,TEM} for a detailed discussion
on the theory of attractors.

The existence of the global attractor in the weak-energy space
$\H^0$ (where $u\in L^2(\Omega)$) has been proved in~\cite{Amnesia},
generalizing some earlier results from~\cite{GMP}. However, both its
finite fractal dimension and the existence of exponential attractors
are established only for $p<3$. It should be noted that, without
growth conditions on $\varphi$ other than~\eqref{critical} (e.g.\
the same polynomial control rate from above and below), the case
$p=3$ is critical, which explains the difficulties faced by
\cite{Amnesia}.

In this work, we are mainly interested to solutions in the higher
regularity space $\H^1$ (where $u\in H_0^1(\Omega)$). Here, the
treatment of the case $p=3$ is even more delicate, since the same
problems encountered in~\cite{Amnesia} arise from the very
beginning. Besides, our assumptions on the memory kernel $\mu$ are
more general (as shown in~\cite{CHEP}, the most general within the
class of decreasing kernels). The strategy to deal with the critical
case leans on an instantaneous regularization of $u$, obtained by
means of estimates of ``hyperbolic" flavor, demanding in turn a
skillful treatment of the memory terms. The effect of such a
regularization is to render the nonlinearity subcritical in all
respects, allowing to construct regular exponentially attracting
sets in $\H^1$. Incidentally, once the existence of global and
exponential attractors in $\H^1$ is established, it is standard
matter to recover analogous results in the less regular space
$\H^0$, extending the analysis of~\cite{Amnesia} to the critical
case $p=3$.

\subsection{Plan of the paper}

The functional setting is introduced in the next Section~2. In
Section~3, we recall some known facts on the solution semigroup. The
main result are then stated in Section~4. The rest of the paper is
devoted to the proofs: in Section~5, we study an auxiliary problem,
which will be used in the subsequent Section~6, in order to draw the
existence of a strongly continuous semigroup in a more regular
space; in Section~7, we demonstrate the existence of a regular
exponentially attracting set, while the final Section~8 contains the
conclusions of the proofs.
%%%%%%%%%%%%%%%%%%%%%%%%%%%%%%%%%%%%%%%%%%%%%%

%%%%%%%%%%%%%%%%%%%%%%%%%%%%%%%%%%%%%%%%%%%%%%
\section{Functional Setting and Notation}

Throughout this work, $\II(\cdot)$ will stand for a {\it generic}
increasing positive function.

Given a Hilbert space $\H$, we denote by $\langle\cdot,\cdot
\rangle_\H$ and $\|\cdot\|_\H$ its inner product and norm, and we
call $\LL(\H)$ the Banach space of bounded linear operators on $\H$.
For $R>0$, we put
$$\Ball_\H(R)=\{z\in\H : \|z\|_{\H}\leq R\}.$$
The Hausdorff semidistance between two sets $\X,\Y\subset\H$ is
defined as
$$\dist_\H(\X,\Y)=\sup_{x\in\X}\inf_{y\in\Y}\|x-y\|_\H,$$
while the fractal dimension of a (relatively) compact set
$\K\subset\H$ is
$$\dim_{\H}(\K)=\limsup_{\eps\to 0^+}
\frac{ \ln {\mathfrak N}_\eps(\K)}{\ln (1/\eps)},$$ ${\mathfrak
N}_\eps(\K)$ being the smallest number of $\eps$-balls of $\H$
necessary to cover $\K$.

We consider the strictly positive Laplace-Dirichlet operator on
$L^2(\Omega)$
$$A=-\Delta,\quad\D(A)=H^2(\Omega)\cap H^1_0(\Omega),
$$
generating, for $r\in\R$, the scale of Hilbert spaces (we omit the
index $r$ when $r=0$)
$$H^r={\D}(A^{r/2}),\quad \l u,v\r_{r}=\l
A^{r/2}u,A^{r/2}v\r_{L^2(\Omega)}.$$ In particular, $H=L^2(\Omega)$,
$H^1=H_0^1(\Omega)$, $H^2=H^2(\Omega)\cap H^1_0(\Omega)$. Whenever
$r_1>r_2$, the embedding $H^{r_1}\subset H^{r_2}$ is compact and
$$\|u\|_{r_1}\geq\lambda_1^{(r_1-r_2)/2}\|u\|_{r_2},\quad\forall u\in
H^{r_1}.$$ Next, we introduce the {\it memory spaces}
$$
\M^r=L^2_\mu(\R^+;H^{r}), \quad \l \eta,\psi\r_{\M^r}=
\int_0^\infty\mu(s)\l\eta(s),\psi(s)\r_{r}\,\d s,
$$
along with the infinitesimal generator of the right-translation
semigroup on $\M^r$, i.e.\ the linear operator
$$T\eta=-\eta',\quad \D_r(T)=\big\{\eta\in{\M^r}:\eta'\in\M^r,\,\,
\eta(0)=0\big\},$$ where the {\it prime} stands for the
distributional derivative, and $\eta(0)=\lim_{s\to 0}\eta(s)$ in
$H^r$. For every $\eta\in\D_r(T)$, we have the basic inequality (see
\cite{Terreni})
\begin{equation}
\label{TTT} \l T \eta,\eta\r_{\M^r}\leq 0.
\end{equation}
Finally, we define the phase spaces
$$
\H^r = H^r \times \M^{r+1}, \quad\V=H^2\times\M^2.$$

\subsection*{A word of warning}

Without explicit mention, we will perform  several formal estimates,
to be in a position to exploit \eqref{TTT}, for instance. As usual,
the estimates are justified within a proper Galerkin approximation
scheme.
%%%%%%%%%%%%%%%%%%%%%%%%%%%%%%%%%%%%%%%%%%%%%%

%%%%%%%%%%%%%%%%%%%%%%%%%%%%%%%%%%%%%%%%%%%%%%
\section{The Semigroup}
\label{SSemi}

Introducing the auxiliary variable
$\eta=\eta^t(s):[0,\infty)\times\R^+\to\R$, accounting for the {\it
integrated past history} of $u$, and formally defined as (see
\cite{DAF,Terreni})
\begin{equation}
\label{AUXV} \eta^t(s)=\int_{0}^s u(t-y)\d y,
\end{equation}
we recast \eqref{ORI}-\eqref{ORI-IC} in the {\it history space
framework}. This amounts to considering the Cauchy problem in the
unknowns $u=u(t)$ and $\eta=\eta^t$
\begin{equation}
\label{SYS}
\begin{cases}
\displaystyle
\pt u+A  u + \int_0^\infty \mu(s)A\eta(s)\,\d s + \varphi(u)=f,\\
\pt \eta=T\eta+ u, \\
\noalign{\vskip1.5mm} (u(0),\eta^0)=z,
\end{cases}
\end{equation}
where $z=(u_0,\eta_0)$ and $f\in H$ is independent of time.

\begin{remark}
The original problem~\eqref{ORI}-\eqref{ORI-IC} is recovered by
choosing
$$u_0=\hat u(0),\quad \eta_0(s)=\int_{0}^s \hat u(-y)\d y.
$$
We address the reader to \cite{CMP,Terreni} for more details on the
equivalence between the two formulations, which, within the proper
functional setting, is not merely  formal.
\end{remark}

Problem~\eqref{SYS} generates a (strongly continuous) semigroup of
solutions $S(t)$, on both the phase spaces $\H^0$ and $\H^1$ (see,
e.g.\ \cite{Amnesia}). Thus,
$$(u(t),\eta^t)=S(t)z.$$
In particular, $\eta^t$ has the explicit representation
formula~\cite{Terreni}
\begin{equation}
\label{REP} \eta^t(s)=
\begin{cases}
\int_0^s u(t-y)\d y & s\leq t,\\
\eta_0(s-t)+\int_0^t u(t-y)\d y & s>t.
\end{cases}
\end{equation}

\begin{remark}
\label{DECLIN} As shown in \cite{CMP}, the linear homogeneous
version of \eqref{SYS}, namely,
\begin{equation}
\label{LINS}
\begin{cases}
\displaystyle
\pt u+A  u + \int_0^\infty \mu(s)A\eta(s)\,\d s=0,\\
\pt \eta=T\eta+ u,
\end{cases}
\end{equation}
generates a strongly continuous semigroup $L(t)$ of (linear)
contractions on {\it every} space $\H^r$, satisfying, for some
$M\geq 1$, $\eps>0$ independent of $r$, the exponential decay
\begin{equation}
\label{DECCO2} \|L(t)\|_{\LL(\H^r)}\leq M \e^{-\eps t}.
\end{equation}
\end{remark}

Let us briefly recall some known facts from \cite{Amnesia}.

\begin{itemize}
\item For $\imath=0,1$, there exists $R_\imath>0$ such that
$$\BB_\imath:=\Ball_{\H^\imath}(R_\imath)$$
is an absorbing set for $S(t)$ in $\H^\imath$.
\smallskip
\item Bounded sets $\B\subset\H^0$
are exponentially attracted by $\BB_1$ in the norm of $\H^0$:
\begin{equation}
\label{LOWATT} \dist_{\H^0}(S(t)\B,\BB_1)\leq
\II\big(\|\B\|_{\H^0}\big)\e^{-\eps_0 t},
\end{equation}
for some $\eps_0>0$.
\smallskip
\item For every $z\in\Ball_{\H^1}(R)$,
\begin{equation}
\label{palmiro} \|S(t)z\|_{\H^1}^2 + \int_t^{t+1}\|
u(\tau)\|_2^2\,\d\tau\leq \II(R).
\end{equation}
\end{itemize}

\begin{remark}
These results have been obtained under the commonly adopted
assumption
\begin{equation}
\label{BAD} \mu'(s)+\delta\mu(s)\leq 0,\quad \text{a.e.}\,\, s>0,
\end{equation}
for some $\delta>0$. On the other hand, it is not hard to show
that~\eqref{NEC} can be equivalently written as
\begin{equation}
\label{NEC2} \mu(s+\sigma)\leq C\e^{-\delta \sigma}\mu(s),\quad
\text{a.e.}\,\, s>0,\,\,\,\forall\sigma\geq 0,
\end{equation}
for some $\delta>0$ and $C\geq 1$, which is easily seen to coincide
with \eqref{BAD} when $C=1$. However, if $C>1$, the gap between
\eqref{BAD} and \eqref{NEC2} is quite relevant (see \cite{CHEP} for
a detailed discussion). For instance, \eqref{BAD} does not allow
$\mu$ to have (even local) flat zones. Besides, any compactly
supported $\mu$ fulfills \eqref{NEC2}, but it clearly need not
satisfy \eqref{BAD}. Nonetheless, the aforementioned results remain
true within \eqref{NEC}, although the proofs require the
introduction of a suitable functional in order to reconstruct the
energy, as in the case of the following Lemma~\ref{LEMMASUPER}.
\end{remark}
%%%%%%%%%%%%%%%%%%%%%%%%%%%%%%%%%%%%%%%%%%%%%%

%%%%%%%%%%%%%%%%%%%%%%%%%%%%%%%%%%%%%%%%%%%%%%
\section{Main Results}
\label{SM}

Defining the vector
\begin{equation}z_f=(u_f,\eta_f)\in H^2\times \D_2(T)\subset \V,
\label{zetaf} \end{equation} with $u_f=\frac12A^{-1}f$ and
$\eta_f(s)=u_f s$, the main result of the paper reads as follows.

\begin{theorem}
\label{MAIN} There exists a compact set $\EE\subset\V$ with
$\dim_\V(\EE)<\infty$, and positively invariant under the action of
$S(t)$, satisfying the exponential attraction property
$$\dist_{\V}(S(t) \B,\EE)
\leq\II\big(\|\B\|_{\H^1}\big) \frac{\e^{-\omega_1
t}}{\sqrt{t}},\quad\forall t>0,
$$
for some $\omega_1>0$ and every bounded set $\B\subset\H^1$.
Moreover,
$$\EE = z_f+ \EE_\star,$$
where $\EE_\star$ is a bounded subset of $\H^3$, whose second
component belongs to $\D_2(T)$.
\end{theorem}

\begin{remark}
The theorem implicitly makes a quite interesting assertion: whenever
$t>0$ and $\B\subset\H^1$ is bounded, $S(t)\B$ is a bounded subset
of $\V$ (cf.\ Proposition~\ref{propREG} below).
\end{remark}

Such a set $\EE$ is called an {\it exponential attractor}. It is
worth noting that, as $\eta_f\in\D_2(T)$, the second component of
$\EE$ belongs to $\D_2(T)$ as well. Besides, if $f\in H^1$, it is
immediate to deduce the boundedness of $\EE$ in $\H^3$.

\begin{corollary}
\label{MAINcorcor} With respect to the Hausdorff semidistance in
$\H^1$, the attraction property improves to
$$\dist_{\H^1}(S(t) \B,\EE)
\leq\II\big(\|\B\|_{\H^1}\big)\e^{-\omega_1 t}.
$$
\end{corollary}

As a byproduct, we establish the existence of the $(\H^1,\V)$-global
attractor.

\begin{theorem}
\label{THMATT} There exists a compact set $\AA\subset\EE$ with
$\dim_\V(\AA)<\infty$, and strictly invariant under the action of
$S(t)$, such that
$$\lim_{t\to\infty}\big[\dist_{\V}(S(t) \B,\AA)\big]=0,
$$
for every bounded set $\B\subset\H^1$.
\end{theorem}

As observed in \cite{Amnesia}, the semigroup $S(t)$ fulfills the
backward uniqueness property on the attractor (in fact, on the whole
space $\H^0$), a typical feature of equations with memory. A
straightforward consequence is

\begin{corollary}
\label{BU} The restriction of $S(t)$ on $\AA$ is a group of
operators.
\end{corollary}

The next result provides the link between the two components of the
solutions on the attractor. Recall that the attractor is made by the
sections (say, at time $t=0$) of all complete bounded trajectories
of the semigroup (see, e.g.\ \cite{HAR}).

\begin{proposition}
\label{CHAR} Any solution $(u(t),\eta^t)$ lying on $\AA$
satisfies~\eqref{AUXV} for all $t\in\R$.
\end{proposition}

\begin{remark}
In particular, we obtain the uniform estimates
$$\sup_{t\in\R}\|\eta^t(s)\|_2\leq c_0s\quad\text{and}\quad
\sup_{t\in\R}\sup_{s>0}\|(\eta^t)'(s)\|_2\leq c_0,
$$
for every $(u(t),\eta^t)$ lying on the attractor, with
$c_0=\sup\{\|u_0\|_2 : (u_0,\eta_0)\in\AA\}$.
\end{remark}

We now focus our attention on $S(t)$ as a semigroup on the phase
space $\H^0$. Indeed, the set $\EE$ of Theorem~\ref{MAIN} turns out
to be an exponential attractor on $\H^0$ as well.

\begin{corollary}
\label{CORMAIN} We have
$$\dist_{\H^0}(S(t) \B,\EE)
\leq\II\big(\|\B\|_{\H^0}\big) \e^{-\omega_0 t},
$$
for some $\omega_0>0$ and every bounded set $\B\subset\H^0$.
\end{corollary}

\begin{corollary}
\label{CORATT} The set $\AA$ is also the global attractor for the
semigroup $S(t)$ on $\H^0$; namely,
$$\lim_{t\to\infty}\big[\dist_{\H^0}(S(t) \B,\AA)\big]=0,
$$
whenever $\B$ is a bounded subset of $\H^0$.
\end{corollary}

The remaining of the paper is devoted to the proofs of the results.
%%%%%%%%%%%%%%%%%%%%%%%%%%%%%%%%%%%%%%%%%%%%%%

%%%%%%%%%%%%%%%%%%%%%%%%%%%%%%%%%%%%%%%%%%%%%%
\section{An Auxiliary Problem}

This section deals with the analysis of the Cauchy problem in the
variable $Z(t)=(u(t),\eta^t)$
\begin{equation}
\label{SYSLEMMA}
\begin{cases}
\displaystyle
\pt u+A  u + \int_0^\infty \mu(s)A\eta(s)\,\d s =f+g,\\
\pt \eta=T\eta+ u, \\
\noalign{\vskip1.5mm} Z(0)=z,
\end{cases}
\end{equation}
where $z=(u_0,\eta_0)\in\H^1$, $f\in H$ is independent of time and
$g\in L^2_{\rm loc}(\R^+;H^1)$.

\smallskip
We need a definition and a preliminary lemma.

\begin{definition}
A nonnegative function $\Lambda$ on $\R^+$ is said to be {\it
translation bounded} if
$$
\t(\Lambda):=\sup_{t\geq 0}\int_t^{t+1}
\Lambda(\tau)\,\d\tau<\infty.
$$
\end{definition}

\begin{lemma}
\label{LEMMAGRNWGEN} For $\imath=0,1,2$, let $\Lambda_\imath$ be
nonnegative functions such that $\t(\Lambda_\imath)\leq m_\imath$.
Assuming $\Lambda_0$ absolutely continuous, let the differential
inequality
$$
\frac{\d}{\d t} \Lambda_0 \leq \Lambda_0\Lambda_1+\Lambda_2
$$
hold almost everywhere in $\R^+$. Then, for every $t\geq 0$,
$$\textstyle\Lambda_0(t)\leq \e^{m_1}\Lambda_0(0)\e^{- t}+
\frac{ \e^{m_1}(m_0+m_0m_1+m_2)}{1-\e^{-1}}.$$
\end{lemma}

\begin{proof}
Setting
$$\tilde\Lambda_1(t)=\Lambda_1(t)-1-m_1,\quad
\tilde\Lambda_2(t)=(1+m_1)\Lambda_0(t)+\Lambda_2(t),$$ we rewrite
the differential inequality as
$$
\frac{\d}{\d t} \Lambda_0 \leq
\Lambda_0\tilde\Lambda_1+\tilde\Lambda_2.
$$
Observe that
$$
\int_\tau^{t}\tilde\Lambda_1(s)\,\d s\leq -(t-\tau)+m_1,\quad\forall
t\geq\tau,
$$
and
$$
\t(\tilde\Lambda_2)\leq m_0+m_0m_1+m_2.
$$
Hence, an application of the Gronwall lemma entails
$$\Lambda_0(t)\leq \e^{m_1}\Lambda_0(0)\e^{-t}+
\e^{m_1}\int_0^t\e^{-(t-\tau)}\tilde\Lambda_2(\tau)\,\d\tau,$$ and
the inequality (cf.\ \cite{CPV})
$$
\int_0^t\e^{-(t-\tau)}\tilde\Lambda_2(\tau)\,\d\tau \leq
\textstyle\frac{1}{1-\e^{-1}}\;\t(\tilde\Lambda_2)
$$
yields the desired result.
\end{proof}

\smallskip

Given $Z=(u,\eta)\in\V$ and $f\in H$, we define the functional
$$\Lambda[Z,f]=
\|u\|_2^2+\alpha\|Z\|_{\H^1}^2+2\l\eta,u \r_{\M^2}-2\l f,Au\r
+8\|f\|^2,
$$
with $\alpha>0$ large enough such that
\begin{equation}
\label{QLE} \textstyle\frac12\|Z\|_\V^2\leq \Lambda[Z,f]\leq
2\alpha\|Z\|_\V^2+\alpha\|f\|^2.
\end{equation}
Moreover, given $u\in L^2_{\rm loc}(\R^+;H^2)$, we set
$$F(t;u)=\int_0^t\mu(t-s)\|u(s)\|_2^2\,\d s.$$

\begin{remark}
\label{remQUIM} Exchanging the order of integration, we have
\begin{equation}
\label{zippo} \t\big(F(\cdot,u)\big)\leq\kappa_0
\t\big(\|u\|_2^2\big).
\end{equation}
\end{remark}

We now state and prove several results on the solution $Z(t)$ to
problem~\eqref{SYSLEMMA}.

\begin{lemma}
\label{QWYM} There is a structural constant $\alpha>0$, large enough
to comply with \eqref{QLE}, such that the functional
$$\Lambda(t)=\Lambda[Z(t),f]$$
satisfies (within the approximation scheme) the differential
inequality
\begin{equation}
\label{KEYIN} \frac{\d}{\d t} \Lambda(t) \leq \mu(t)\Lambda(t)+\th
F(t;u) +\th\|f\|^2+\th\|g(t)\|_1^2,
\end{equation}
for some positive constant $\th=\th(\alpha)$.
\end{lemma}

\begin{proof}
Multiplying the first equation of \eqref{SYSLEMMA} by $A\pt u$, and
using the second equation, we obtain the differential equality
$$
\frac{\d}{\d t} \big\{\|u\|_2^2+2\l\eta,u \r_{\M^2}-2\l f,Au
\r\big\} + 2\|\pt u\|^2_1= 2\kappa_0\|u\|_2^2+ 2 \l T\eta,u
\r_{\M^2} + 2 \l g, \pt u\r_1.
$$
Arguing exactly as in \cite[Lemma 4.3]{DPZ}, we find $\alpha>0$,
depending only on the total mass $\kappa_0$ of $\mu$, such that
\begin{equation}
\label{VRR} 2\l T\eta, u\r_{\M^2}\leq \|u\|_2^2+\mu\|u\|_2^2+\alpha
F -2\alpha\l T \eta,\eta\r_{\M^2}.
\end{equation}
Clearly, due to \eqref{TTT}, the estimate is still valid for a
larger $\alpha$. Thus, controlling the last term as
$$2 \l g, \pt u\r_1\leq \|g\|_1^2+\|\pt u\|^2_1,$$
we end up with
\begin{align}
\label{TOPO}
&\frac{\d}{\d t} \big\{\|u\|_2^2+2\l\eta,u \r_{\M^2}-2\l f,Au \r\big\}\\
&\leq (1+2\kappa_0)\|u\|_2^2 +\mu\|u\|_2^2+\alpha F -2\alpha\l T
\eta,\eta\r_{\M^2}+ \|g\|_1^2.\notag
\end{align}
A further multiplication of~\eqref{SYSLEMMA} by $Z$ in $\H^1$
entails
$$
\frac{\d}{\d t}\|Z\|_{\H^1}^2 +2\|u\|_2^2 =2\l T
\eta,\eta\r_{\M^2}+2\l f,Au \r+ 2 \l g,u\r_1.
$$
Exploiting the straightforward relation
$$2\l f,Au \r+ 2 \l g,u\r_1\leq
\|u\|_2^2+2\|f\|^2+2\lambda_1^{-1}\|g\|_1^2,$$ we are led to the
inequality
\begin{equation}
\label{LINO} \frac{\d}{\d t}\|Z\|_{\H^1}^2+\|u\|_2^2 \leq 2\l T
\eta,\eta\r_{\M^2}+2\|f\|^2+2\lambda_1^{-1}\|g\|_1^2.
\end{equation}
We now choose $\alpha\geq 1+2\kappa_0$ such that \eqref{QLE} and
\eqref{VRR} hold. Adding \eqref{TOPO} and $\alpha$-times
\eqref{LINO}, we finally get \eqref{KEYIN}.
\end{proof}

\begin{lemma}
\label{QWYM1} Assume that
$$\t\big(\|Z\|_\V^2\big)\leq \beta\quad\text{and}\quad
\|g(t)\|_1\leq \gamma\big(1+\|Z(t)\|_\V^2\big),$$ for some
$\beta,\gamma>0$. Then, there exists $D=D(\beta,\gamma,\|f\|)>0$
such that
$$\|Z(t)\|_\V\leq D\Big(\frac{1}{\sqrt{t}}+1\Big).
$$
If $z\in\V$, the estimate improves to
$$\|Z(t)\|_\V\leq D\|z\|_\V\,\e^{-t}+D.
$$
\end{lemma}

\begin{proof} From \eqref{QLE}, it is readily seen that
$$\|g\|_1\leq \gamma+2\gamma\Lambda.$$
Thus, defining
\begin{align*}
\Lambda_1(t)&=\mu(t)+2\th\gamma\|g(t)\|_1,\\
\Lambda_2(t)&=\th F(u;t)+\th\|f\|^2+2\th\gamma\|g(t)\|_1,
\end{align*}
inequality \eqref{KEYIN} turns into
$$
\frac{\d}{\d t} \Lambda \leq \Lambda\Lambda_1+\Lambda_2.
$$
Using again \eqref{QLE}, and recalling \eqref{zippo}, we learn that
$$\t(\Lambda)+\t(\Lambda_1)+\t(\Lambda_2)\leq C,$$
for some $C>0$ depending (besides on $\kappa_0$) only on
$\beta,\gamma,\|f\|$. Hence, Lemma~\ref{LEMMAGRNWGEN} together with
a further application of \eqref{QLE} entail the second assertion of
the lemma. If $z\not\in\V$, we apply a standard trick: we set
$$\tilde\Lambda(t)=\frac{t}{1+t}\Lambda(t),$$
which satisfies
$$
\frac{\d}{\d t} \tilde\Lambda \leq
\tilde\Lambda\Lambda_1+\tilde\Lambda_2,
$$
where $\tilde\Lambda_2(t)=\Lambda(t)+\Lambda_2(t)$. Note that
$$\t(\tilde\Lambda)+\t(\tilde\Lambda_2)\leq 2C.$$
As in the previous case, the first assertion follows from
Lemma~\ref{LEMMAGRNWGEN} and \eqref{QLE}.
\end{proof}

\begin{lemma}
\label{QWYM2} Suppose that $z\in\V$, $f=0$ and
$$\|g(t)\|_1\leq k\|Z(t)\|_\V,$$
for some $k\geq 0$. Then, there are $D_1>0$ and $D_2=D_2(k)>0$ such
that
$$\|Z(t)\|_\V\leq D_1\|z\|_\V\,\e^{D_2 t}.
$$
\end{lemma}

\begin{proof}
Under these assumptions, \eqref{QLE} and \eqref{KEYIN} become
$$
\textstyle\frac12\|Z\|_\V^2\leq \Lambda\leq 2\alpha\|Z\|_\V^2
$$
and
$$
\frac{\d}{\d t} \Lambda \leq (\mu+2\th\alpha k^2)\Lambda+\th F.
$$
Moreover, exchanging the order of integration,
$$\int_0^{t} F(\tau)\,\d\tau\leq \kappa_0\int_0^{t}\|u(\tau)\|_2^2\,\d\tau
\leq 2\kappa_0\int_0^{t} \Lambda(\tau)\,\d\tau.
$$
Hence, integrating the differential inequality on $(0,t)$, we arrive
at
$$
\Lambda (t) \leq \Lambda(0) + \int_0^t \big[\mu(\tau) + 2\th(\alpha
k^2 + \kappa_0) \Lambda (\tau)\big] \, \d \tau.
$$
Making use of the integral Gronwall lemma,
$$
\|Z(t)\|^2_\V \leq 2\Lambda (t) \leq  2\Lambda(0) \e^{\kappa_0}
\e^{2\th(\alpha k^2 + \kappa_0) t}  \leq 4\alpha \e^{\kappa_0}
\|z\|^2_{\V} \e^{2\th(\alpha k^2 + \kappa_0) t},  $$ and the result
follows by choosing $D_1= 4\alpha \e^{\kappa_0}$ and
$D_2=2\th(\alpha k^2 + \kappa_0)$.
\end{proof}
%%%%%%%%%%%%%%%%%%%%%%%%%%%%%%%%%%%%%%%%%%%%%%

%%%%%%%%%%%%%%%%%%%%%%%%%%%%%%%%%%%%%%%%%%%%%%
\section{The Semigroup on $\V$}

We begin with a suitable regularization property for the solutions
departing from $\H^1$.

\begin{proposition}
\label{propREG} Let  $z\in\Ball_{\H^1}(R)$. Then, for every $t>0$,
$S(t)z\in\V$ and the estimate
$$\|S(t)z\|_\V\leq \II(R) \Big(\frac{1}{\sqrt{t}}+1\Big)
$$
holds. If in addition $z\in \V$,
$$\|S(t)z\|_\V\leq \II(R) \|z\|_\V\,\e^{-t}+\II(R).
$$
\end{proposition}

\begin{proof}
We know from \eqref{palmiro} that the solution $Z(t)=S(t)z$ fulfills
$$\t\big(\|Z\|_\V^2\big)\leq \II(R),$$
whereas \eqref{critical}, \eqref{palmiro} and the Agmon inequality
\begin{equation}
\label{AGMON} \|u\|_{L^\infty(\Omega)}^2\leq c_\Omega\|u\|_1\|u\|_2
\end{equation}
entail
$$
\|\varphi(u)\|_1=\|\varphi'(u)\nabla u\|\leq
\|\varphi'(u)\|_{L^\infty(\Omega)}\|u\|_1 \leq
\II(R)\big(1+\|z(t)\|_\V^2\big).
$$
Hence, Lemma~\ref{QWYM1} with $g=-\varphi(u)$ applies.
\end{proof}

\begin{corollary}
\label{corREG} There exists $R_\V>0$ such that the set
$$\BB_\V:=\Ball_\V(R_\V)
$$
has the following property: for every $R>0$ there is a time
$t_\V=t_\V(R)>0$ such that
$$S(t)\Ball_{\H^1}(R)\subset\BB_\V,\quad \forall t\geq t_\V.$$
\end{corollary}

\begin{proof}
Let $z\in \BB_1$. According to Proposition~\ref{propREG},
$$\|S(t)z\|_\V\leq \II(R_1)\Big(\frac{1}{\sqrt{t}}+1\Big).
$$
Thus, setting $R_\V=2\II(R_1)$, the inclusion
$S(t)\BB_1\subset\BB_\V$ holds for every $t\geq 1$. Since $\BB_1$ is
absorbing in $\H^1$, for every $R>0$ there exists $t_1=t_1(R)$ such
that $S(t)\Ball_{\H^1}(R)\subset\BB_1$ whenever $t\geq t_1$. We
conclude that
$$S(t)\Ball_{\H^1}(R)\subset S(t-t_1)\BB_1\subset\BB_\V,\quad \forall
t\geq t_\V,$$ with $t_\V=t_1+1$.
\end{proof}

In particular, Proposition~\ref{propREG} tells that $S(t)$ is a
semigroup on $\V$, which, by Corollary~\ref{corREG}, possesses the
absorbing set $\BB_\V$. In fact, $S(t)$ is a strongly continuous
semigroup, as the next proposition shows.

\begin{proposition}
\label{propSTR1} For $\imath=1,2$, let $z_\imath\in \Ball_{\V}(R)$.
Then, we have the continuous dependence estimate
$$\|S(t)z_1-S(t)z_2\|_\V\leq D_1\|z_1-z_2\|_\V
\,\e^{\II(R) t}.
$$
\end{proposition}

\begin{proof}
Calling $(u_\imath(t),\eta^t_\imath)=S(t)z_\imath$, the difference
$(\bar u(t),\bar\eta^t)=S(t)z_1-S(t)z_2$ fulfills the problem
$$
\begin{cases}
\displaystyle \pt \bar u+A \bar u + \int_0^\infty
\mu(s)A\bar\eta(s)\,\d s
=\varphi(u_2)-\varphi(u_1),\\
\pt \bar\eta=T\bar\eta+ \bar u, \\
\noalign{\vskip1.5mm} \bar z(0)=z_1-z_2.
\end{cases}
$$
Due to Proposition~\ref{propREG}, $\|u_\imath\|_2\leq\II(R)$.
Exploiting \eqref{critical} and the Agmon inequality~\eqref{AGMON},
it is then immediate to see that
 $$\|\varphi(u_2)-\varphi(u_1)\|_1\leq \II(R)\|\bar u\|_2,
$$
and the claim is a consequence of Lemma~\ref{QWYM2} with $f=0$ and
$g=\varphi(u_2)-\varphi(u_1)$.
\end{proof}

\begin{proposition}
\label{propSTR2} For every fixed $z\in\V$,
$$t\mapsto S(t)z\in\C([0,\infty),\V).$$
\end{proposition}

\begin{proof}
Let $\tau>0$ be fixed. Given $z\in\V$, choose a regular sequence
$z_n\to z$ in $\V$, such that $t\mapsto
S(t)z_n\in\C([0,\infty),\V)$. For every $n,m\in\N$,
Proposition~\ref{propSTR1} provides the estimate
$$\sup_{t\in[0,\tau]}\|S(t)z_n-S(t)z_m\|_\V\leq C\|z_n-z_m\|_\V,
$$
for some $C>0$ depending on $\tau$ and on the $\V$-bound of $z_n$.
Therefore, $t\mapsto S(t)z_n$ is a Cauchy sequence in
$\C([0,\tau],\V)$. Accordingly, its limit $t\mapsto S(t)z$ belongs
to $\C([0,\tau],\V)$. Since $\tau>0$ is arbitrary, we are done.
\end{proof}

Finally, we dwell on the linear homogeneous case, that is,
system~\eqref{LINS}. From the previous results, we know that $L(t)$
is a strongly continuous semigroup of linear operators on $\V$. We
prove that $L(t)$ is exponentially stable as well.

\begin{proposition}
\label{propExpStab} The semigroup $L(t)$ satisfies the exponential
decay property
\begin{equation}
\label{DECCO3} \|L(t)\|_{\LL(\V)}\leq M_1 \e^{-\eps_1 t},
\end{equation}
for some $M_1\geq 1$ and $\eps_1>0$.
\end{proposition}

\begin{proof}
Let $z\in\V$. By virtue of \eqref{DECCO2} we have that
$$\|\eta^t\|_{\M^2}\leq \|L(t)z\|_{\H^1}\leq M \e^{-\eps t}\|z\|_{\H^1}.$$
Thus,
$$\int_0^\infty \|\eta^\tau\|_{\M^2}^2\,\d\tau<\infty.
$$
On the other hand, multiplying \eqref{LINS} times $(u,\eta)$ in
$\H^1$, and using \eqref{TTT}, we get
$$\frac{\d}{\d t}\|S(t)z\|^2_{\H^1}+2\|u(t)\|_2^2\leq 0.$$
Integrating the inequality, we obtain
$$\int_0^\infty \|u(\tau)\|_2^2\,\d\tau\leq
\textstyle\frac12\|z\|_{\H^1}^2<\infty.$$ We conclude that
$$\int_0^\infty \|L(\tau)z\|_\V^2\,\d\tau<\infty,\quad\forall z\in\V,$$
and the result follows from the celebrated theorem of R.\ Datko
\cite{DAT} (see also \cite{PAZ}).
\end{proof}
%%%%%%%%%%%%%%%%%%%%%%%%%%%%%%%%%%%%%%%%%%%%%%

%%%%%%%%%%%%%%%%%%%%%%%%%%%%%%%%%%%%%%%%%%%%%%
\section{Regular Exponentially Attracting Sets}

\subsection{The result}

We show the existence of a compact subset of $\V$ which
exponentially attracts $\BB_\V$, with respect to the Hausdorff
semidistance in $\V$. To this end, we introduce the further
space~\cite{Amnesia}
$$
\W=\big\{\eta \in \M^4\cap\D_2(T): \,\,\Xi[\eta]< \infty\big\},
$$
where
$$\Xi[\eta]=\|T\eta\|_{\M^2}^2
+  \sup_{x \geq 1} \bigg[x\int_{(0,1/x) \cup (x,\infty)}\mu(s)
\|\eta(s)\|^2_{2}\,\d s\bigg].$$ This is a Banach space endowed with
the norm
$$
\|\eta\|_{\W}^2 = \|\eta\|_{\M^4}^2 +\Xi[\eta].
$$
Finally, we define the product space
$$\Z=H^3\times\W\subset\H^3.$$

\begin{remark}
By means of a slight generalization of~\cite[Lemma 5.5]{PZ}, the
embedding $\Z\subset\V$ is compact (this is the reason why $\Z$ is
needed), contrary to the embedding $\H^3\subset\V$, which is clearly
continuous, but never compact. Moreover, closed balls of $\Z$ are
compact in $\V$ (see~\cite{CGGP}).
\end{remark}

\begin{theorem}
Let $z_f$ be given by \eqref{zetaf}. \label{ATT1} There exists
$R_\star>0$ such that
$$\BB:=z_f+\Ball_{\Z}(R_\star)$$
fulfills the following properties:
\begin{itemize}
\item[(i)] There is $t_\star=t_\star(R_\star)>0$ such that
$S(t)\BB\subset\BB$ for every $t\geq t_\star$.
\item[(ii)] The inequality
$$\dist_{\V}(S(t)\BB_\V,\BB)
\leq C_1 \e^{-\eps_1 t}
$$
holds for some $C_1>0$, with $\eps_1$ as in \eqref{DECCO3}.
\end{itemize}
\end{theorem}

Theorem~\ref{ATT1} is a consequence of the next lemma, proved in
Subsection~\ref{subsub}.

\begin{lemma}
\label{LEMMASUPER} Let $\II_\imath(\cdot)$ denote generic increasing
positive functions. For every $z\in\Ball_\V(R)$, the semigroup
$S(t)z$ admits the decomposition
$$S(t)z=z_f+\ell_1(t;z)+\ell_2(t;z),$$
where
\begin{align}
\label{UN}
\|\ell_1(t;z)\|_{\V}&\leq \II_1(R)\e^{-\eps_1 t},\\
\label{DU} \|\ell_2(t;z)\|_{\Z}&\leq \II_2(R).
\end{align}
If in addition $z\in z_f+\Ball_{\Z}(\varrho)$, we have the further
estimate
\begin{equation}
\label{TRI} \|\ell_1(t;z)\|_{\Z}\leq \II_3(\varrho)\e^{-\eps_2
t}+\II_4(R),
\end{equation}
for some $\eps_2>0$.
\end{lemma}

\begin{proof}[Proof of Theorem \ref{ATT1}]
For any given $ R, \varrho >0$ and
$$z\in\Ball_\V(R)\cap[z_f+\Ball_{\Z}(\varrho)],
$$
it is readily seen from \eqref{DU}-\eqref{TRI} that
\begin{equation}
\label{XRT} \|S(t)z-z_f\|_\Z\leq \II_3(\varrho)\e^{-\eps_2
t}+\II_2(R)+\II_4(R).
\end{equation}
We fix then $\BB$ by selecting
$$R_\star=2\II_2(R_\V)+2\II_4(R_\V),$$
with $R_\V$ as in Corollary~\ref{corREG}. In particular, defining
$$\varrho_\star = \II_3(R_\star)
+\II_2(\|\BB\|_\V)+\II_4(\|\BB\|_\V),$$ inequality~\eqref{XRT}
provides the inclusion
$$
S(t)\BB\subset z_f+\Ball_\Z(\varrho_\star),\quad\forall t\geq 0.
$$
On the other hand, by Corollary~\ref{corREG}, there is a time
$t_\e\geq 0$ (the entering time of $\BB_\V$ into itself) for which
$$S(t_\e)\BB\subset S(t_\e)\BB_\V\subset\BB_\V=\Ball_\V(R_\V).$$
In conclusion,
$$S(t_\e)\BB\in\Ball_\V(R_\V)\cap[z_f+\Ball_{\Z}(\varrho_\star)],
$$
and a further application of \eqref{XRT} for $t\geq t_\e$ leads to
$$\|S(t)z-z_f\|_\Z\leq \II_3(\varrho_\star)\e^{-\eps_2 (t-t_\e)}
\textstyle +\frac{1}{2}R_\star,\quad\forall z\in\BB.
$$
Accordingly, (i) holds true by taking a sufficiently large
$t_\star=t_\star(R_\star)\geq t_\e$. Finally, since
$\II_2(R_\V)<R_\star$, relations \eqref{UN}-\eqref{DU} immediately
entail the estimate
$$\dist_{\V}(S(t)\BB_\V,\BB)
\leq \II_1(R_\V)\e^{-\eps_1 t},
$$
establishing (ii).
\end{proof}

\subsection{Proof of Lemma \ref{LEMMASUPER}}

\label{subsub} We will make use of the following technical lemma
(see~\cite{Amnesia} for a proof).

\begin{lemma}
\label{TECH1} Given $\eta_0\in\W$ and $u\in L^\infty_{\rm
loc}(\R^+;H^2)$, let $\eta=\eta^t(s)$ be the unique solution to the
Cauchy problem in $\M^2$
$$
\begin{cases}
\pt\eta^t=T\eta^t+ u(t),\\
\eta^0=\eta_0.
\end{cases}
$$
Then, $\eta^t\in\D_2(T)$ for every $t>0$, and
$$\Xi[\eta^t]\leq Q^2\,\Xi[\eta_0]\e^{-2\nu
t}+Q^2\|u\|_{L^\infty(0,t;H^2)}^2,
$$
for some $Q\geq 1$ and some $\nu>0$, both independent of $\eta_0$
and $u$.
\end{lemma}

In the sequel, $C>0$ will denote a generic constant, which may
depend (increasingly) only on $R$. Given $z\in\Ball_\V(R)$, we put
$$\ell_1(t;z)=L(t)(z-z_f),\quad\ell_2(t;z)=W(t),$$
where, by comparison, the function $W(t)=(w(t),\xi^t)$ solves the
problem
\begin{equation}
\label{SYS2}
\begin{cases}
\displaystyle
\pt w+A  w + \int_0^\infty \mu(s)A\xi(s)\d s + \varphi(u)=0,\\
\pt \xi=T\xi+ w, \\
\noalign{\vskip1.5mm} W(0)=0.
\end{cases}
\end{equation}
In light of Proposition~\ref{propExpStab}, we get at once
\eqref{UN}. Indeed,
$$
\|L(t)(z-z_f)\|_{\V}\leq M_1\|z-z_f\|_\V\,\e^{-\eps_1 t}\leq
C\e^{-\eps_1 t}.
$$
If $z\in z_f+\Ball_{\Z}(\varrho)$, the decay property \eqref{DECCO2}
provides the estimate
$$
\|L(t)(z-z_f)\|_{\H^3}\leq M\|z-z_f\|_{\H^3}\e^{-\eps t} \leq
M\varrho\e^{-\eps t}.
$$
The second component of $L(t)(z-z_f)=(v(t),\psi^t)$ fulfills the
problem
$$
\begin{cases}
\pt\psi^t=T\psi^t+ v,\\
\psi^0=\eta_0-\eta_f,
\end{cases}
$$
and the $\V$-estimate above ensures the uniform bound
$$\|v(t)\|_2\leq C.$$
Therefore, by Lemma~\ref{TECH1},
$$\Xi[\psi^t]\leq Q^2\,\Xi[\eta_0-\eta_f]\e^{-2\nu t}+C
\leq Q^2\varrho^2\e^{-2\nu t}+C.
$$
Putting $\eps_2=\min\{\eps,\nu\}$, we obtain
$$
\|L(t)(z-z_f)\|_{\Z}^2=\|L(t)(z-z_f)\|_{\H^3}^2+\Xi[\psi^t] \leq
(M^2+Q^2)\varrho\e^{-2\eps_2 t}.
$$
This proves~\eqref{TRI}. We now turn to system \eqref{SYS2}. Thanks
to Proposition~\ref{propREG},
$$\|u(t)\|_2\leq C.$$
By \eqref{critical}, it is then standard matter to verify that
$$\|\varphi(u(t))\|_2\leq C.$$
Multiplying \eqref{SYS2} by $W$ in $\H^3$, and using \eqref{TTT}, we
arrive at
$$\frac{\d}{\d t}\|W\|_{\H^3}^2+2\|w\|^2_4\leq
\|\varphi(u)\|_2\|w\|_4\leq \|w\|_4^2+C.
$$
In order to reconstruct the energy, following~\cite{CMP}, we
introduce the functional
$$\Upsilon(t) = \int_0^\infty k(s) \|\xi^t(s)\|_4^2\,\d s,$$
which, in light of \eqref{NEC}, satisfies the bound
$$\Upsilon\leq\Theta\|\xi\|_{\M^4}^2,$$
and the differential inequality
\begin{align*}
\frac{\d}{\d t}\Upsilon &=-\|\xi\|_{\M^4}^2+2\int_0^\infty
k(s)\l\xi(s),w\r_4\,\d s\\
&\leq-\|\xi\|_{\M^4}^2+2\Theta\|w\|_4\int_0^\infty
\mu(s)\|\xi(s)\|_4\,\d s
\textstyle\leq-\frac12\|\xi\|_{\M^4}^2+\kappa_0\Theta^2\|w\|_4^2.
\end{align*}
Defining then
$$\Psi(t)=\Theta_0\|W(t)\|_{\H^3}^2+\Upsilon(t),$$
for $\Theta_0>\max\{\Theta,\kappa_0\Theta^2\}$ (so that, in
particular, $\Psi$ and $\|W\|_{\H^3}^2$ control each other) the
differential inequality
$$\frac{\d}{\d t}\Psi+\varpi\Psi\leq C,
$$
holds for some $\varpi=\varpi(\Theta_0,\Theta,\lambda_1)>0.$ Hence,
the Gronwall lemma gives the uniform bound
$$\|W(t)\|_{\H^3}\leq C.$$
Finally, applying Lemma~\ref{TECH1} to the second equation
of~\eqref{SYS2}, we get
$$\Xi[\xi^t]\leq C.$$
Summarizing,
$$\|W(t)\|_\Z\leq C.$$
This establishes \eqref{DU} and completes the proof of the lemma.
\qed
%%%%%%%%%%%%%%%%%%%%%%%%%%%%%%%%%%%%%%%%%%%%%%

%%%%%%%%%%%%%%%%%%%%%%%%%%%%%%%%%%%%%%%%%%%%%%
\section{Exponential Attractors}

The next step is to demonstrate the existence of a regular set $\EE$
which exponentially attracts $\BB_\V$.

\begin{theorem}
\label{EXPMAINTHM} There exists a compact set $\EE\subset\V$ with
$\dim_\V(\EE)<\infty$, and positively invariant for $S(t)$, such
that
$$
\dist_{\V}(S(t) \BB_\V,\EE) \leq C_\star \e^{-\omega_1 t},
$$
for some $C_\star>0$ and some $\omega_1>0$.
\end{theorem}

We preliminary observe that, thanks to the exponential decay
property of Theorem~\ref{ATT1}
$$\dist_{\V}(S(t)\BB_\V,\BB)
\leq C_1 \e^{-\eps_1 t},
$$
and the continuous dependence estimate provided by
Proposition~\ref{propSTR1}, the transitivity of the exponential
attraction, devised in~\cite{FGMZ}, applies. Hence, it suffices to
prove the existence of a set $\EE$ complying with the statement of
the theorem, but satisfying only the weaker exponential decay
estimate
\begin{equation}
\label{EXPtotal} \dist_{\V}(S(t) \BB,\EE) \leq C_0 \e^{-\omega t},
\end{equation}
for some $C_0>0$ and some $\omega>0$. Thus, in light of the abstract
result from~\cite{EMZ} on the existence of exponential attractors
for discrete semigroups in Banach spaces, and thereafter
constructing the attractor for the continuous case in a standard
way, Theorem~\ref{EXPMAINTHM} applies provided that we show the
following facts:
\begin{itemize}
\item[(i)] There exist positive functions $\gamma(\cdot)$ and
$\Gamma(\cdot)$, with $\gamma$ vanishing at infinity, such that the
decomposition
$$S(t)z_1-S(t)z_2=\ell_1(t;z_1,z_2)+\ell_2(t;z_1,z_2),
$$
holds for every $z_1,z_2\in\BB$, where
\begin{align*}
\|\ell_1(t;z_1,z_2)\|_{\V}&\leq \gamma(t)\|z_1-z_2\|_\V,\\
\|\ell_2(t;z_1,z_2)\|_{\Z}&\leq \Gamma(t)\|z_1-z_2\|_\V.
\end{align*}
\item[(ii)] There exists $K\geq 0$ such that
$$\sup_{z\in \BB}\|S(t)z-S(\tau)z\|_\V\leq K|t-\tau|,
\quad\forall t,\tau\in[t_\star,2t_\star].
$$
\end{itemize}
Indeed, recalling that $\BB$ is closed in $\V$, by means of (i) we
obtain the existence of an exponential attractor $\EE_{\rm
d}\subset\BB$ for the discrete semigroup $S_n:=S(n
t_\star):\BB\to\BB$. Then, we define
$$\EE=\bigcup_{t\in
[t_\star,2 t_\star]}S(t)\EE_{\rm d}.$$ Due to (ii) and
Proposition~\ref{propSTR1}, the map
$$(t,z)\mapsto S(t)z:[t_\star,2t_\star]\times\BB\to\BB,$$
is Lipschitz continuous with respect to the
$(\R\times\V,\V)$-topology. This guarantees that $\EE$ shares the
same features of $\EE_{\rm d}$ (e.g.\ positive invariance and finite
fractal dimension).

\begin{proof}[Proof of {\rm (i)}]
Till the end of the section, the generic constant $C>0$ depends only
on $\BB$. Setting $S(t)z_\imath=(u_\imath(t),\eta_\imath^t)$ and
$\bar z=z_1-z_2$, we write
$$S(t)z_1-S(t)z_2=L(t)\bar z+W(t),
$$
where $W(t)=(\bar w(t),\bar\xi^t)$ solves the problem
\begin{equation}
\label{SYS2bis}
\begin{cases}
\displaystyle \pt \bar w+A \bar w + \int_0^\infty
\mu(s)A\bar\xi(s)\d s
=\varphi(u_2)-\varphi(u_1),\\
\pt\bar\xi=T\bar\xi+\bar w, \\
\noalign{\vskip1.5mm} W(0)=0.
\end{cases}
\end{equation}
By Proposition~\ref{propExpStab},
$$
\|L(t)\bar z\|_{\V}\leq M_1\|\bar z\|_\V\,\e^{-\eps_1 t}.
$$
Taking advantage of \eqref{critical} and Proposition~\ref{propSTR1},
$$\|\varphi(u_2(t))-\varphi(u_1(t))\|_2\leq C \|u_2(t)-u_1(t)\|_2
\leq C \|\bar z\|_\V\,\e^{C t}.
$$
Hence, multiplying \eqref{SYS2bis} by $W$ in $\H^3$, and using
\eqref{TTT}, we obtain
$$\frac{\d}{\d t}\|W\|_{\H^3}^2\leq C \|\bar z\|_\V^2\,\e^{C t},
$$
and an integration in time readily gives
$$\|W(t)\|_{\H^3}^2\leq C \|\bar z\|_\V^2\,\e^{C t}.
$$
Accordingly, from Lemma~\ref{TECH1} applied to the second equation
of~\eqref{SYS2bis},
$$\Xi[\bar\xi^t]\leq C \|\bar z\|_\V^2\,\e^{C t}.
$$
Consequently, we learn that
$$\|W(t)\|_{\Z}\leq C \|\bar z\|_\V\,\e^{C t}.
$$
Therefore, (i) holds with the choice $\ell_1(t;z_1,z_2)=L(t)\bar z$
and $\ell_2(t;z_1,z_2)=W(t)$.
\end{proof}

\begin{proof}[Proof of {\rm (ii)}]
We will show that
$$\sup_{t\in[t_\star,2t_\star]}\sup_{z\in \BB}\|\pt S(t)z\|_\V\leq C,
$$
which clearly implies (ii). For $z=(u_0,\eta_0)\in\BB$, the function
$(\tilde u(t),\tilde\eta^t)=\pt S(t) z$ fulfills the Cauchy problem
$$
\begin{cases}
\displaystyle \pt \tilde u+A \tilde u + \int_0^\infty
\mu(s)A\tilde\eta(s)\,\d s
+\varphi'(u)\tilde u=0,\\
\pt \tilde\eta=T\tilde\eta+\tilde u, \\
\noalign{\vskip1.5mm} (\tilde u(0),\tilde\eta^0)=\tilde z,
\end{cases}
$$
where
$$\textstyle
\tilde z=(-Au_0-\int_0^\infty \mu(s)A\eta_0(s)\,\d s
-\varphi(u_0)+f,T\eta_0+ u_0).$$ Observe that
$$\|\varphi'(u)\tilde u\|_1\leq C\|\tilde u\|_2.$$
Thus, applying Lemma~\ref{QWYM2} with $f=0$ and
$g=-\varphi'(u)\tilde u$, and noting that $\|\tilde z\|_{\V}\leq C$,
the claim follows.
\end{proof}
%%%%%%%%%%%%%%%%%%%%%%%%%%%%%%%%%%%%%%%%%%%%%%

%%%%%%%%%%%%%%%%%%%%%%%%%%%%%%%%%%%%%%%%%%%%%%
\section{Proofs of the Main Results}

We have now all the ingredients to carry out the proofs of the
results stated in Section~\ref{SM}.

\begin{proof}[Proofs of Theorem \ref{MAIN} and Corollary
\ref{MAINcorcor}] Let $\B\subset\Ball_{\H^1}(R)$, for some $R>0$.
According to Corollary~\ref{corREG}, there is a positive time
$t_\V=t_\V(R)$ such that
$$S(t)\B\subset\BB_\V,\quad \forall t\geq t_\V.$$
Therefore, by Theorem~\ref{EXPMAINTHM},
$$
\dist_{\V}(S(t) \B,\EE) \leq \II(R)\e^{-\omega_1 t},\quad \forall
t\geq t_\V.
$$
On the other hand, by virtue of Proposition~\ref{propREG},
$$
\dist_{\V}(S(t) \B,\EE) \leq \frac{\II(R)}{\sqrt{t}},\quad \forall
t\in(0,t_\V).
$$
Collecting the two inequalities we obtain
$$\dist_{\V}(S(t) \B,\EE)
\leq\II(R) \frac{\e^{-\omega_1  t}}{\sqrt{t}},\quad\forall t>0.
$$
The remaining properties of $\EE$ are ensured by
Theorem~\ref{EXPMAINTHM}. With respect to the Hausdorff semidistance
in $\H^1$, we have
$$\dist_{\H^1}(S(t) \B,\EE)\leq
\big(\lambda_1^{-1/2}+1\big)\dist_{\V}(S(t) \B,\EE) \leq\II(R)
\e^{-\omega_1 t},\quad\forall t\geq 1,
$$
and, due to \eqref{palmiro},
$$\dist_{\H^1}(S(t) \B,\EE)
\leq\II(R),\quad\forall t< 1.
$$
Hence, Corollary~\ref{MAINcorcor} follows.
\end{proof}

Theorem~\ref{THMATT} is a direct consequence of Theorem~\ref{MAIN}
(cf.\ \cite{BV,TEM}).

\begin{proof}[Proof of Proposition \ref{CHAR}]
Let $Z(t)=(u(t),\eta^t)$ be a solution lying on $\AA$. Assume first
$t>0$. Fixed an arbitrary $\tau>0$, denote $z_\tau=S(-\tau)Z(0)$ and
set
$$(u_\tau(t),\eta^t_\tau)=S(t)z_\tau.$$
Observing that
$$(u_\tau(t+\tau),\eta^{t+\tau}_\tau)=(u(t),\eta^t),$$
the representation formula~\eqref{REP} for $\eta^{t+\tau}$ gives
$$
\eta^t(s)=\eta^{t+\tau}_\tau(s)=\int_{0}^s u_\tau(t+\tau-y)\d y
=\int_{0}^s u(t-y)\d y,
$$
whenever $0<s\leq t+\tau$. From the arbitrariness of $\tau>0$, we
conclude that \eqref{AUXV} is valid for all $t> 0$. If $t\leq 0$,
the argument is similar, and left to the reader.
\end{proof}

\begin{proof}[Proof of Corollary \ref{CORMAIN}]
Let $\B\subset\Ball_{\H^0}(R)$, for some $R>0$. From \eqref{LOWATT},
$$\dist_{\H^0}(S(t)\B,\BB_1)\leq \II(R)\e^{-\eps_0 t},
$$
whereas Corollary~\ref{MAINcorcor} implies, in particular, that
$$\dist_{\H^0}(S(t) \BB_1,\EE)
\leq\II(R_1)\e^{-\omega_1 t}.
$$
Besides, exploiting \eqref{DERIVATA}, the continuous dependence
estimate
$$
\|S(t)z_1-S(t)z_2\|_{\H^0} \leq \e^{c t}\|z_1-z_2\|_{\H^0}
$$
is easily seen to hold for some $c>0$ and every $z_1,z_2\in\H^0$.
Once again, we take advantage of the transitivity of the exponential
attraction~\cite{FGMZ}, and we obtain the required exponential
attraction property.
\end{proof}

Similarly to the case of Theorem~\ref{THMATT},
Corollary~\ref{CORATT} is a byproduct of Corollary \ref{CORMAIN} and
of the $\V$-regularity of the (exponentially) attracting set.
%%%%%%%%%%%%%%%%%%%%%%%%%%%%%%%%%%%%%%%%%%%%%%

\medskip

%%%%%%%%%%%%%%%%%%%%%%%%%%%%%%%%%%%%%%%%%%%%
\subsection*{Acknowledgments}

The authors are grateful to Professor Roger Temam for the unique
environment he provided to perform this work, at the Institute of
Scientific Computing and Applied Mathematics, Indiana University.
MDC greatly acknowledges the Mathematics Department of Indiana
University for hospitality and support.
%%%%%%%%%%%%%%%%%%%%%%%%%%%%%%%%%%%%%%%%%%%%

%%%%%%%%%%%%%%%%%%%%%%%%%%%%%%%%%%%%%%%%%%%%%%

%%%%%%%%%%%%%%%%%%%%%%%%%%%%%%%%%%%%%%%%%%%%%%

\end{document}